\documentclass[11pt]{amsart}
\usepackage{latexsym,dsfont}
\usepackage{amsmath,amsthm,amsfonts,amssymb,enumerate,color,url,hyperref}
\usepackage[top=30mm,right=30mm,bottom=30mm,left=30mm]{geometry}

\newlength{\standardunitlength}
\newtheorem{prop}{Proposition}[section]

\newtheorem{lemma}[prop]{Lemma}
\newtheorem{cor}[prop]{Corollary}
\newtheorem{theorem}[prop]{Theorem}

\newtheorem{remark}[prop]{Remark}

\numberwithin{equation}{section}
\renewcommand{\setminus}{\smallsetminus}

\newcommand{\F}{\mathbb{F}}

\newcommand{\GL}{\mathrm{GL}}

\newcommand{\GU}{\mathrm{GU}}

\newcommand{\Sp}{\mathrm{Sp}}

\newcommand{\GLnq}{\mathrm{GL}_n(q)}
\newcommand{\GLtwoq}{\mathrm{GL}_2(q)}
\newcommand{\SLtwoq}{\mathrm{SL}_2(q)}
\newcommand{\GUnq}{\mathrm{GU}_n(q)}
\newcommand{\GUkq}{\mathrm{GU}_k(q)}
\newcommand{\GUiq}{\mathrm{GU}_i(q)}
\newcommand{\Spnq}{\mathrm{Sp}_{2n}(q)}
\newcommand{\Spkq}{\mathrm{Sp}_{2k}(q)}
\newcommand{\Spiq}{\mathrm{Sp}_{2i}(q)}
\newcommand{\Spfourq}{\mathrm{Sp}_4(q)}
\newcommand{\PV}{\mathsf{P}}
\newcommand{\QV}{\mathsf{Q}}
\newcommand{\EV}{\mathsf{E}}
\newcommand{\Var}{\mathsf{Var}}
\newcommand{\Irr}{\mathrm{Irr}}
\newcommand{\triv}{{\mathds{1}}}
\newcommand{\Id}{\mathrm{Id}}

\begin{document}

\title{Tensor product Markov chains and Weil representations}

\author{Jason Fulman}

\address{Department of Mathematics\\
       University of Southern California\\
       Los Angeles, CA, 90089, USA}
\email{fulman@usc.edu}

\author{Michael Larsen}
\address{Department of Mathematics\\
    Indiana University\\
    Bloomington, IN 47405\\
    U.S.A.}
\email{mjlarsen@indiana.edu}

\author{Pham Huu Tiep}
\address{Department of Mathematics\\ Rutgers University\\ Piscataway, NJ 08854\\ U.S.A.}
\email{tiep@math.rutgers.edu}

\thanks{The first author was supported by Simons Foundation grant 917224. The second author was supported by NSF grant DMS-2401098.
The third author gratefully acknowledges the support of the NSF (grant DMS-2200850) and the Joshua Barlaz Chair in Mathematics.}



\date{July 16, 2024}

\begin{abstract} We obtain sharp bounds on the convergence rate of Markov chains on irreducible representations of
finite general linear, unitary, and symplectic groups (in both odd and even characteristic) given by tensoring with Weil representations. 
\end{abstract}

\maketitle

\tableofcontents

\section{Introduction} In recent years there has been much work on convergence rates of tensor product Markov chains:
see \cite{BDLP}, \cite{F}, \cite{F2}, \cite{LT}, and the references therein. As explained in these papers, tensor product chains are connected to
many parts of mathematics: to the Burnside-Brauer theorem for building irreducible representations, to the McKay correspondence and McKay graphs,
to Pitman's 2M-X theorem, to quantum computing, and to using Stein's method to study Plancherel measure of finite groups. 

These tensor product Markov chains are defined as follows. The state space is $\Irr(G)$, the set of irreducible representations 
of a finite group $G$. Let $\eta$ be a fixed representation of $G$ (not necessarily irreducible or with a real-valued
character). Letting $d_{\lambda}$ be the
dimension of a representation $\lambda$ and letting $m_{\rho}(\lambda \otimes \eta)$ denote the multiplicity of $\rho$ in
the tensor product $\lambda \otimes \eta$, this Markov chain has transition probabilities
\begin{equation} \label{tensdefine} K(\lambda,\rho) = \frac{m_{\rho}(\lambda \otimes \eta) d_{\rho}}{d_{\lambda} d_{\eta}}.
\end{equation} The stationary distribution $\pi$ of $K$ in our examples is the Plancherel measure of $G$,
defined by $\pi(\rho)=\frac{d_{\rho}^2}{|G|}$.

In this paper we study the case that $\eta$ is a (reducible) {\it Weil representation} of a finite general linear, unitary, or symplectic
group. These representations were studied by A. Weil 
\cite{Weil} for classical groups over local fields, following work of Segal and Shale. Weil mentioned that 
the finite field case may be considered analogously. This was developed in 
detail by R. E. Howe \cite{Ho} and P. G\'erardin \cite{Ger}, for characteristic zero
representations. The same representations, still in characteristic zero, 
were introduced independently by 
I. M. Isaacs \cite{I1} and H. N. Ward \cite{Ward} 
for symplectic groups $\Spnq$ with $q$ odd, and by G. M. Seitz \cite{Sei} for unitary groups $\GUnq$. (These representations for $\mathrm{Sp}_{2n}(p)$ were also constructed in \cite{BRW}.) 
Weil representations of finite symplectic groups $\mathrm{Sp}_{2n}(q)$ with $2|q$ were constructed 
by R. M. Guralnick and the third author in \cite{GT}.
Weil representations attract much attention 
because of their many interesting features, cf. for instance 
\cite{Gr}, \cite{T}, \cite{TZ}, \cite{Zal}. 

The tensor product chains of Weil representations
can be diagonalized using representation theory (see \cite{F} or \cite{BDLP}), but this is not too helpful in the case that
$\eta$ is not real-valued.

In Section \ref{main} we provide a self-contained approach to convergence rates of tensor product chains which does not
rely on this diagonalization (see Theorem \ref{charbound}). For all the Weil representations $\eta$ studied here,
the absolute value of the corresponding character value $\chi^{\eta}(g)$ depends only on the dimension of the fixed space of $g$ (that is
the kernel of $g-1$). Combining this
with work of Rudvalis and Shinoda \cite{RS} on the proportion of elements in a finite classical group with fixed space of
a given dimension, we are able to get upper bounds on the convergence rate of the tensor product chain. To see that
these bounds are sharp (that is to get matching lower bounds), we need to study properties of the product of two
uniformly chosen random transvections. This turns out to be intricate, especially for odd characteristic symplectic groups
(for which transvections split into two conjugacy classes). We remark that our bounds show that these Markov chains
have a cutoff in the sense of \cite{Dia}.

We mention that in the case of $\GLnq$, the convergence rate of the Weil representation was analyzed in \cite{F}.
The calculations in our upper bound is similar, but the method in \cite{F} for studying the lower bound does
not seem to extend to other finite classical groups. So in this paper we provide a different approach to lower bounds,
based on studying the product of two random transvections. This approach does work for other finite classical groups.

In this paper we do not treat finite orthogonal groups. Our method of proof should carry over, but the calculations are
very similar, tedious, and repetitive (one needs to treat plus/minus type, odd/even dimension, and odd/even characteristic).
We also note that, in principle, the {\it level theory} as developed in \cite{GLT}, \cite{GLT2} also allows one to work out some upper and
lower bounds of convergence rates of tensor product chains, which however may not be as strong as the bounds proved in this paper.
 
To close the introduction, we mention a bit of notation. If $\PV$ is a probability measure on a finite set $X$
and $A$ is a subset of $X$, we let $\PV(A)$ be the probability of $A$. If $W$ is a random variable, we let $\EV_{\PV}(W)$
denote the expected value of $W$ under the probability distribution $P$, and we let $\Var_{\PV}(W)$ denote the variance
of $W$ under the probability distribution $P$.  We remind the reader that $\Var_{\PV}(W) = \EV_{\PV}(W^2) - (\EV_{\PV}(W))^2$.
Also, when proving lower bounds for convergence rates, we will frequently use Chebyshev's inequality, which states
that 
$$\PV(|W-\EV(W)| \geq k) \leq \frac{\Var(W)}{k^2}.$$ 
Some notation about Markov chains will appear in the next section.

\section{Tensor product Markov chains: general bounds} \label{main}

This section derives a general convergence rate upper bound for tensor product Markov chains and recalls a tool that
will be useful for obtaining lower bounds.

The total variation distance between probability distributions $\PV$ and $\QV$ on a finite set $X$ is given by
\[ ||\PV-\QV||_{\mathsf{TV}} = \frac{1}{2} \sum_{x \in X} |\PV(x)-\QV(x)| = \max_{A \subseteq X} |\PV(A)-\QV(A)|.\] We let $K_x^r$ be the distribution on $X$ after $r$ steps of the Markov chain $K$ started from $x$. (Our Markov chains will all be on the set of irreducible representations of a finite classical group, and will be started at the
trivial representation $\triv$, so we usually write $K^r$ instead of $K^r_{\triv}$). We  let $K^r(x,y)$ denote the chance of going from $x$ to $y$ in $r$ steps. We always let $\pi$ denote the stationary distribution of the Markov chain $K$. A distribution $\pi$ is called {\it stationary} if
\[ \pi(y) = \sum_x \pi(x) K(x,y) \] for all $y$. Since $X$ is finite, stationary distributions exist but are not always unique (see Chapter 1 of \cite{LPW}).
For our examples,
$\pi$ is Plancherel measure (and is unique) and $K^r \rightarrow \pi$
as $r \rightarrow \infty$. The Markov chain $K$ is said to be {\it reversible} if 
\[ \pi(x) K(x,y) = \pi(y) K(y,x) \] for all $x$ and $y$. The Markov chains studied in this paper are not reversible in
general, but Theorem \ref{charbound} shows that for our purposes they behave as if they were reversible.

We begin with a general known elementary bound and then consider tensor product Markov chains.

\begin{prop} \label{general} Let $K$ be a Markov chain on a finite set $X$ with stationary distribution $\pi$
satisfying the property that $\pi(y)>0$ for all $y$.

\begin{enumerate}[\rm(i)]
\item 
\[ 4 ||K_x^r - \pi||_{\mathsf{TV}}^2 \leq \sum_y \frac{K^r(x,y)^2}{\pi(y)} - 1.\]
\item
If the chain is reversible, then
\[ 4 ||K_x^r - \pi||_{\mathsf{TV}}^2 \leq \frac{K^{2r}(x,x)}{\pi(x)} - 1.\]
\end{enumerate}
\end{prop}

\begin{proof} For the first assertion,
\begin{eqnarray*}
2 ||K_x^r - \pi||_{\mathsf{TV}} & = & \sum_y |K^r(x,y) - \pi(y)| \\
& = & \sum_y \frac{|K^r(x,y) - \pi(y)|}{\sqrt{\pi(y)}} \sqrt{\pi(y)}.
\end{eqnarray*} 

By the Cauchy-Schwarz inequality this is at most
\[ \sqrt{\sum_y \frac{|K^r(x,y) - \pi(y)|^2}{\pi(y)}} \sqrt{\sum_y \pi(y)} = \sqrt{\sum_y
 \frac{|K^r(x,y) - \pi(y)|^2}{\pi(y)}}.\] 
Squaring gives that
\[ 4 ||K_x^r - \pi||_{\mathsf{TV}}^2 \leq \sum_y
 \frac{|K^r(x,y) - \pi(y)|^2}{\pi(y)} =  \sum_y \frac{K^r(x,y)^2}{\pi(y)} - 1.\] 

For the second assertion, observe that
\begin{eqnarray*}
\sum_y \frac{K^r(x,y)^2}{\pi(y)} - 1 & = & \sum_y K^r(x,y) \frac{K^r(x,y)}{\pi(y)} -1 \\
& = & \sum_y K^r(x,y) \frac{K^r(y,x)}{\pi(x)} - 1 \\
& = & \frac{K^{2r}(x,x)}{\pi(x)} - 1,
\end{eqnarray*} where the second equality used reversibility.
\end{proof} 

We now specialize to the tensor product Markov chain $K$ defined in \eqref{tensdefine}.

Lemma \ref{used0} collects two useful facts. Note that it is not necessary to assume that $K$ is reversible.

\begin{lemma} \label{used0}
\begin{enumerate}[\rm(i)]
\item Letting $\triv$ denote the trivial representation, one has that
\begin{equation} \label{obvious}  K^r(\triv,\rho) = \frac{d_{\rho} m_{\rho} (\eta^r)}{d_{\eta}^r},
\end{equation}
where $m_{\rho} (\eta^r)$ denotes the multiplicity of $\rho$ in $\eta^r$.

\item A stationary distribution of this Markov chain is given by Plancherel measure
of the group $G$: \[ \pi(\rho) = \frac{d_{\rho}^2}{|G|}.\] 
\end{enumerate}
\end{lemma}

\begin{proof} For the first part we use induction on $r$. If $\chi^\rho$ denotes the character of the representation $\rho$, then we have
\begin{eqnarray*}
K^r(\triv,\rho) & = & \sum_{\tau} K^{r-1}(\triv,\tau) K(\tau,\rho) \\
& = & \sum_{\tau} \frac{d_{\tau} m_{\tau}(\eta^{r-1})}{d_{\eta}^{r-1}} \frac{d_{\rho} m_{\rho}(\eta \otimes \tau)}{d_{\tau}
d_{\eta}}\\
& = & \frac{d_{\rho}}{d_{\eta}^r} \sum_{\tau} m_{\tau}(\eta^{r-1}) m_{\rho}(\eta \otimes \tau) \\
& = & \frac{d_{\rho}}{d_{\eta}^r} \sum_{\tau} \frac{1}{|G|} \sum_{g \in G} \chi^{\eta}(g)^{r-1} \overline{\chi^{\tau(g)}}
\frac{1}{|G|} \sum_{h \in G} \chi^{\eta}(h) \chi^{\tau}(h) \overline{\chi^{\rho}(h)} \\
& = & \frac{d_{\rho}}{d_{\eta}^r} \frac{1}{|G|} \frac{1}{|G|} \sum_{g \in G} \sum_{h \in G} \chi^{\eta}(g)^{r-1} \chi^{\eta}(h)
\overline{\chi^{\rho(h)}} \sum_{\tau} \overline{\chi^{\tau}(g)} \chi^{\tau}(h).
\end{eqnarray*} Using the orthogonality relation for irreducible characters and letting $C(g)$ denote the centralizer of $g$, this is equal to
\begin{eqnarray*}
 \frac{d_{\rho}}{d_{\eta}^r} \frac{1}{|G|} \sum_{g \in G} \frac{1}{|C(g)|} \chi^{\eta}(g)^{r-1} \chi^{\eta}(g) \overline{\chi^{\rho}(g)} |C(g)| & = & \frac{d_{\rho}}{d_{\eta}^r} \frac{1}{|G|} \sum_{g \in G} \chi^{\eta}(g)^{r} \overline{\chi^{\rho}(g)} \\
& = &  \frac{d_{\rho} m_{\rho} (\eta^r)}{d_{\eta}^r}.
\end{eqnarray*}

For the second part, it is necessary to verify that
\[ \sum_{\lambda} \pi(\lambda) K(\lambda,\rho) = \pi(\rho).\]
One calculates that
\begin{eqnarray*}
\sum_{\lambda} \pi(\lambda) K(\lambda,\rho) & = & \sum_{\lambda} \frac{d_{\lambda}^2}{|G|} \left[ \frac{1}{|G|} \sum_{g} \chi^{\lambda}(g) \chi^{\eta}(g)
\overline{\chi^{\rho}(g)} \right] \frac{d_{\rho}}{d_{\lambda} d_{\eta}} \\
& = & \frac{d_{\rho}}{d_{\eta}} \frac{1}{|G|} \sum_g \chi^{\eta}(g) \overline{\chi^{\rho}(g)} \frac{1}{|G|} \sum_{\lambda} d_{\lambda}
\chi^{\lambda}(g).
\end{eqnarray*} By the orthogonality relation for irreducible characters, this simplies to $d_{\rho}^2/|G|$.
\end{proof}

Note that in the next theorem, $K$ does not have to be reversible.

\begin{theorem} \label{charbound} Consider the Markov chain on irreducible representations of $G$ given by tensoring
with $\eta$ and starting at the trivial representation $\triv$. Then \[ 4 ||K^r_{\triv} - \pi ||_{\mathsf{TV}}^2 \leq \sum_{C \neq \{\Id\}} |C|
\left| \frac{\chi^{\eta}(C)}{d_{\eta}} \right|^{2r} ,\] where the sum is over all 
non-identity conjugacy classes $C$ of $G$.
\end{theorem}

\begin{proof} By \eqref{obvious},
\[ K^r(\triv,\rho) = \frac{d_{\rho}}{d_{\eta}^r} \frac{1}{|G|} \sum_g \chi^{\eta}(g)^r
\overline{\chi^{\rho}(g)}.\]

Hence
\begin{eqnarray*}
 \sum_{\rho} \frac{K^r(\triv,\rho)^2}{\pi(\rho)}
& = & \sum_{\rho} \frac{|G|}{d_{\rho}^2} \frac{d_{\rho}^2}{d_{\eta}^{2r}} \frac{1}{|G|^2}
\sum_{g_1} [\chi^{\eta}(g_1)]^r \overline{\chi^{\rho}(g_1)}
\sum_{g_2} [\overline{\chi^{\eta}(g_2)}]^r \chi^{\rho}(g_2) \\
& = & \frac{1}{|G|} \frac{1}{d_{\eta}^{2r}} \sum_{\rho} \sum_{C_1} |C_1| [\chi^{\eta}(C_1)]^r
\overline{\chi^{\rho}(C_1)} \sum_{C_2} |C_2| [\overline{\chi^{\eta}(C_2)}]^r \chi^{\rho}(C_2) \\
& = & \frac{1}{d_{\eta}^{2r}} \sum_{C_1} |C_1| \chi^{\eta}(C_1)^r \sum_{C_2} |C_2|
[\overline{\chi^{\eta}(C_2)}]^r \frac{1}{|G|} \sum_{\rho} \overline{\chi^{\rho}(C_1)} \chi^{\rho}(C_2).
\end{eqnarray*} By the orthogonality relation for irreducible characters, this is equal to 
\[ \frac{1}{d_{\eta}^{2r}} \sum_C |C|^2 \chi^{\eta}(C)^r [\overline{\chi^{\eta}(C)}]^r \frac{1}{|C|}
 =  \frac{1}{d_{\eta}^{2r}} \sum_C |C| |\chi^{\eta}(C)|^{2r}.\] 

The result now follows from part 1 of Theorem \ref{general}.
\end{proof}

Finally, we mention the following proposition from \cite{F} which, together with enumerative results about the product of transvections,
will be useful for lower bounding the convergence rates of our Markov chains. A nice survey of techniques for proving lower bounds
for finite Markov chains is given by Saloff-Coste \cite{Sal}.

\begin{prop} \label{combine} Let $C$ be a conjugacy class of a finite group $G$. Let $f_C$ be the function on irreducible
representations of $G$ defined by \[ f_C(\rho) = \frac{|C|^{1/2} \chi^{\rho}(C)}{d_{\rho}}.\] Define $\EV_{K^r}[(f_C)^s]$ to
be the expected value of the function $(f_C)^s$ after $r$ steps of the random walk $K$ started at the trivial representation. Let
$p_{s,C}(T)$ be the probability that the random walk on $G$ generated by $C$ and started at the identity is in the
conjugacy class $T$ after $s$ steps. Then
\[ \EV_{K^r}[(f_C)^s] = |C|^{s/2} \sum_T p_{s,C}(T) \left( \frac{\chi^{\eta}(T)}{d_{\eta}} \right)^r,\] where the
sum is over all conjugacy classes $T$ of $G$.
\end{prop}

We remark that Proposition \ref{combine} does not require that $C=C^{-1}$. However we will apply Chebyshev's inequality
to study the function $f_C$, so it is important that $f_C$ is real-valued, which holds if $C=C^{-1}$, and so for 
the class of transvections for $\GLnq,\GUnq$ and even characteristic $\Spnq$. For odd characteristic $\Spnq$ there are
two conjugacy classes of transvections. If $q$ is congruent to 1 mod 4, and $C$ is one of these classes of transvections,
then $f_C$ is real-valued and Proposition \ref{combine} will be helpful. If $q$ is congruent to 3 mod 4, then $f_C$ is not
real-valued, so we will prove a modification of Proposition \ref{combine} and study the function $f_*$ defined by
\[ f_*(\rho)= \frac{|C_1|^{1/2} \chi^{\rho}(C_1)}{\sqrt{2} d_{\rho}} +  \frac{|C_2|^{1/2} \chi^{\rho}(C_2)}{\sqrt{2} d_{\rho}},\]
where $C_1,C_2$ are the two conjugacy classes of transvections. Note that $f_*$ is real-valued since $C_2=C_1^{-1}$.

\section{Weil representations of general linear groups}

The paper \cite{F} studied the tensor product random walk on the irreducible representations of $\GLnq$ given by tensoring
with $\eta$, where $\eta$ has character $q^{\dim(\ker(g-1))}$ (the exponent of $q$ is the dimension of the kernel of $g-1$).
This $\eta$ is called the Weil representation. A main result of that paper of \cite{F} was 
\begin{theorem} \label{sharp}
\begin{enumerate}[\rm(i)]
\item If $r=n+c$ with $c>0$, then $||K_{\triv}^r-\pi||_{\mathsf{TV}} \leq \frac{1}{2q^c}$.
\item If $r=n-c$ with $c>0$, then  $||K_{\triv}^r-\pi||_{\mathsf{TV}} \geq 1 - \frac{a}{q^c}$, where $a$ is a
universal constant (independent of $n,q,c$).
\end{enumerate}
\end{theorem}

Although Theorem \ref{sharp} is a nice result, the proof technique for the lower bound does not seem to easily extend to the other finite
classical groups. Our purpose here is to give a proof which does extend.

\begin{theorem} \label{goalGL1} Let $K_{\triv}^r$ be the distribution on $\Irr(\GLnq)$ given by r steps of tensoring with the
Weil representation started at the trivial representation. Then given $\epsilon>0$, there exists $c>0$ (depending only on $\epsilon$) such that
\[ ||K_{\triv}^r - \pi||_{\mathsf{TV}} \geq 1 - \epsilon \] for $r=n-c$ and sufficiently large $n$. 
\end{theorem}

\begin{proof} Let $C$ denote the conjugacy class of transvections in $\GLnq$, and define the function $f_C$ on $\Irr(\GLnq)$ by
\[ f_C(\rho) = \frac{|C|^{1/2} \chi^{\rho}(C)}{d_{\rho}}.\] Since $C=C^{-1}$, the function $f_C$ is
real-valued, so we can apply Chebyshev's inequality to study it.

Let $A$ be the event that $f_C \leq \alpha$, for $\alpha>0$ to be specified momentarily. The orthogonality
relations for irreducible characters imply that under the Plancherel measure $\pi$, the random variable $f_C$ has
mean 0 and variance 1. So by Chebyshev's inequality, $\pi(A) \geq 1 - \frac{1}{\alpha^2}$. Letting
$\alpha=\frac{2}{3} q^{c-1}$, it follows that
\begin{equation} \label{firstuse} \pi(A) \geq 1 - \frac{2.25}{q^{2c-2}}.
\end{equation}

Next we will lower bound the mean $\EV_{K^r}[f_C]$ and upper bound the variance $\Var_{K^r}(f_C)$. By Proposition
\ref{combine}, \[ \EV_{K^r}[f_C] = |C|^{1/2} \left( \frac{\chi^{\eta}(C)}{d_{\eta}} \right)^r = \sqrt{\frac{(q^n-1)(q^{n-1}-1)}{q-1}} q^{-r}
= \sqrt{\frac{(q^n-1)(q^{n-1}-1)}{q-1}} q^{c-n}.\] For $n \geq 3$, this is at least \[ q^{n-1} q^{c-n} = q^{c-1}.\] 

To study the  variance $\Var_{K^r}(f_C)$, we use Proposition \ref{combine} with $s=2$ and Theorem  \ref{gl-transv}.
We conclude that $\EV_{K^r}[(f_C)^2]$ is equal to $\frac{(q^n-1)(q^{n-1}-1)}{q-1}$ multiplied by
\begin{eqnarray*}
& & \frac{q-1}{(q^n-1)(q^{n-1}-1)} + \left( \frac{1}{q} \right)^r \frac{2q^n-2q^{n-1}-q^2-q+2}{(q^n-1)(q^{n-1}-1)} \\
& & + \left( \frac{1}{q^2} \right)^r \frac{q^{2n-1}-3q^n+q^{n-1}+q^2}{(q^n-1)(q^{n-1}-1)}.
\end{eqnarray*}

Since \[ (\EV_{K^r}[f_C])^2 = \frac{(q^n-1)(q^{n-1}-1)}{q-1} \left( \frac{1}{q^2} \right)^r, \] we conclude that
\[ \Var_{K^r}[f_C] = \EV_{K^r}(f_C^2) - (\EV_{K^r}(f_C))^2 = 1 + T_1 + T_2 - T_3 \] where
\[\begin{aligned} T_1 & = \frac{1}{q^{n-c}} \frac{2q^n-2q^{n-1}-q^2-q+2}{q-1},\\
  T_2 & =  \frac{1}{q^{2n-2c}} \frac{q^{2n-1}-3q^n+q^{n-1}+q^2}{q-1}, \\
 T_3 & = \frac{1}{q^{2n-2c}} \frac{(q^n-1)(q^{n-1}-1)}{q-1}.\end{aligned}\] One easily checks that $T_1 \leq 2q^{c-1}$ and that $T_2<T_3$
if $n \geq 2$. So for such $n$, 
$$ \Var_{K^r}[f_C] \leq 1 + 2 q^{c-1}.$$

Recalling that $\EV_{K^r}[f_C] \geq q^{c-1}$ for $n \geq 3$, we conclude that for such $n$, 
\begin{eqnarray*}
K^r(A)= K^r \left( f_C \leq \frac{2}{3} q^{c-1} \right) & = & K^r \left( f_C - \EV_{K^r}(f_C) \leq \frac{2}{3} q^{c-1} - \EV_{K^r}(f_C) \right) \\
& \leq & K^r \left( f_C - \EV_{K^r}(f_C) \leq \frac{2}{3} q^{c-1} - q^{c-1} \right) \\
& = & K^r \left( f_C - \EV_{K^r}(f_C) \leq - \frac{1}{3} q^{c-1} \right) \\
& \leq & K^r \left( |f_C-\EV_{K^r}(f_C)| \geq \frac{1}{3} q^{c-1} \right). \end{eqnarray*}
By Chebyshev's inequality, this is at most
\[ \frac{ 9 \Var_{K^r}(f_C)}{q^{2c-2}} \leq  \frac{ 9 (1+2q^{c-1})}{q^{2c-2}},\]
and thus
$$K^r(A) \leq  \frac{ 9 (1+2q^{c-1})}{q^{2c-2}}.$$
Since \[ ||K_{\triv}^r-\pi||_{\mathsf{TV}} \geq |K_{\triv}^r(A)-\pi(A)|, \]  
this inequality and \eqref{firstuse} imply the theorem.
\end{proof} 

\begin{remark}
{\em The proof of Theorem \ref{goalGL1} shows that its conclusion holds whenever $n \geq 3$ and 
$$\epsilon \geq 11.25q^{-2c+2}+18q^{-c+1},$$
making the dependence of $\epsilon$ on $c$ explicit. A similar remark applies to further lower bound results in the paper, 
Theorems \ref{goalGU}, \ref{lower1mod4}, \ref{lower3mod4}, and \ref{lower-Sp-even}. 
}
\end{remark}

\section{Weil representations of unitary groups}

Recall by Theorem \ref{charbound} that for the Markov chain on irreducible representations of a finite group $G$
given by tensoring with the representation $\eta$,

\[ 4 ||K^r_{\triv} - \pi ||_{\mathsf{TV}}^2 \leq \sum_{C \neq \{\Id\}} |C|
\left| \frac{\chi^{\eta}(C)}{d_{\eta}} \right|^{2r} .\]

According to \cite[formula (9)]{TZ}, if $\eta$ is the (reducible) Weil representation, then
for any element $g$ of the general unitary group $\GUnq = \mathrm{GU}(n,\F_{q^2})$, \[ \chi^{\eta}(g) = (-1)^n (-q)^{\dim(\ker(g-1))}.\]
In particular, $\chi^{\eta}(\Id)=q^n$. Thus from Theorem \ref{charbound},
\begin{eqnarray*}
4 ||K^r_{\triv} - \pi ||_{\mathsf{TV}}^2 & \leq & 
\sum_{k=0}^{n-1} \sum_{g \atop \dim(\ker(g-1))=k)} \frac{q^{2rk}}{q^{2nr}} \\
& = & \frac{|\GUnq|}{q^{2nr}} \sum_{k=0}^{n-1} q^{2rk} \PV_U(n,k),
\end{eqnarray*} where $\PV_U(n,k)$ is the probability that a random element of
$\GUnq$ has a $k$-dimensional fixed space.

The following theorem is due to Rudvalis and Shinoda \cite{RS}. Recall that
\[ |\GUnq| = q^{{n \choose 2}} \prod_{i=1}^n (q^i-(-1)^i).\] 

\begin{theorem} \label{RSunitary} The proportion $\PV_U(n,k)$ is equal to
\[ \frac{1}{|\GUkq|} \sum_{i=0}^{n-k} \frac{(-1)^i (-q)^{{i \choose 2}}}{(-q)^{ki} |\GUiq|}.\]
\end{theorem}

Corollary \ref{countfixer} gives an upper bound on the quantity in Theorem \ref{RSunitary}.

\begin{cor} \label{countfixer} The number of elements of $\GUnq$ with a k-dimensional fixed space is at most
\[ \left( 1+\frac{1}{q^{k+1}} \right) q^{n^2-k^2}.\]
\end{cor}

\begin{proof} By Theorem \ref{RSunitary} we need to upper bound
\[ \frac{|\GUnq|}{|\GUkq|} \sum_{i=0}^{n-k} \frac{(-1)^i (-q)^{{i \choose 2}}}{(-q)^{ki} |\GUiq|}.\]

First consider the term $\frac{|\GUnq|}{|\GUkq|}$. Using the formula for the size of the unitary groups,
and simplifying, this term is exactly equal to \[ q^{n^2-k^2} \prod_{i=k+1}^n \left( 1 - \left( \frac{-1}{q} \right)^i \right).\]
This is at most $\left( 1+\frac{1}{q^{k+1}} \right) q^{n^2-k^2}$ if $k$ is even and at most $q^{n^2-k^2}$
if $k$ is odd.

Next consider the term
\[ \sum_{i=0}^{n-k} \frac{(-1)^i (-q)^{{i \choose 2}}}{(-q)^{ki} |\GUiq|}.\] Using the formula for the
size of the unitary groups, this is equal to \begin{equation} \label{formU}
\sum_{i=0}^{n-k} \frac{(-1)^i (-q)^{{i \choose 2}}}{(-q)^{ki} q^{{i \choose 2}} \prod_{j=1}^i (q^j-(-1)^j)}.
\end{equation} The first case is that $k$ is even. Then the ith term is positive if $i$ is equal to 0 or 3 mod 4,
and negative if $i$ is equal to 1 or 2 mod 4. Since the terms in the summand are decreasing in magnitude,
it follows that \eqref{formU} is at most 1 if $k$ is even. The second case is that $k$ is odd. Then the ith
term is positive if $i$ is equal to 0 or 1 mod 4, and is negative if $i$ is equal to 2 or 3 mod 4. Since the
terms in the summand are decreasing in magnitude, it follows that for $k$ odd, \eqref{formU} is at most
the sum of its first two terms, which is \[ 1 + \frac{1}{q^k(q+1)} \leq 1 + \frac{1}{q^{k+1}}.\] 

Putting the pieces together, it follow that for all $k$,
\[  \frac{|\GUnq|}{|\GUkq|} \sum_{i=0}^{n-k} \frac{(-1)^i (-q)^{{i \choose 2}}}{(-q)^{ki} |\GUiq|}
\leq \left( 1+\frac{1}{q^{k+1}} \right) q^{n^2-k^2}.\]
\end{proof}

The above ingredients lead to the following theorem.

\begin{theorem} \label{unitbound} Let $K$ be the Markov chain on irreducible representations of $\GUnq$
given by tensoring with the (reducible) Weil representation. Then for $r=n+c$ with $c>0$,
\[ ||K_{\triv}^r - \pi||_{\mathsf{TV}} \leq .7 q^{-c}.\] 
\end{theorem}

\begin{proof} From Theorem \ref{charbound} we know that
\begin{eqnarray*}
4  ||K_{\triv}^r - \pi||_{\mathsf{TV}}^2 & \leq & \sum_{k=0}^{n-1} \PV_U(n,k) |\GUnq| \left( \frac{1}{q^{n-k}} \right)^{2r} \\
& = & \sum_{k=1}^n \PV_U(n,n-k) |\GUnq| \left( \frac{1}{q^k} \right)^{2r}. \end{eqnarray*}

By Corollary \ref{countfixer}, this is at most \[
\sum_{k=1}^n \left( 1+\frac{1}{q^{n-k+1}} \right) q^{n^2-(n-k)^2} \left( \frac{1}{q^k} \right)^{2r} 
\leq \left( 1+\frac{1}{q} \right) \sum_{k=1}^n  q^{n^2-(n-k)^2} \left( \frac{1}{q^k} \right)^{2r}.
\] Since $r=n+c$, this is at most
\[ (1+1/q) \sum_{k=1}^n \frac{1}{q^{k^2+2kc}} \leq  (1+1/q)  \frac{1}{q^{2c}} \sum_{k=1}^n 1/q^{k^2}
\leq \frac{(1+1/q)}{q(1-1/q)} \frac{1}{q^{2c}} \leq 1.5 q^{-2c}.\] 

Hence \[ 4 ||K_{\triv}^r - \pi||_{\mathsf{TV}}^2 \leq  1.5 q^{-2c},\] so the theorem follows by taking square
roots. \end{proof}

Next we lower bound the convergence rate of the tensor product Markov chain. As in the case of $\mathrm{GL}$, we
require an enumerative result about transvections (see Theorem \ref{glutransv}).

\begin{theorem} \label{goalGU} Let $K_{\triv}^r$ be the distribution on $\Irr(\GUnq)$ given by r steps of tensoring with the
Weil representation started at the trivial representation. Then given $\epsilon>0$, there exists $c>0$ (depending only on $\epsilon$) such that
\[ ||K_{\triv}^r - \pi||_{\mathsf{TV}} \geq 1 - \epsilon \] for $r=n-c$ and sufficiently large $n$. 
\end{theorem}

\begin{proof} Let $C$ denote the conjugacy class of transvections in $\GUnq$ and define the function $f_C$ on $\Irr(\GUnq)$ by
\[ f_C(\rho) = \frac{|C|^{1/2} \chi^{\rho}(C)}{d_{\rho}}.\] Since $C=C^{-1}$, the function $f_C$ is real-valued and we can
apply Chebyshev's inequality to it.

Let $A$ be the event that $f_C \leq \frac{1}{4} q^{c-1}$. By Chebyshev's inequality, it follows that
\[ \pi(A) \geq 1 - \frac{16}{q^{2c-2}}.\]

Suppose first that $r$ is even. By Proposition \ref{combine},
\begin{eqnarray*}
\EV_{K^r}(f_C) & = & |C|^{1/2} \left( \frac{\chi^{\eta}(C)}{d_{\eta}} \right)^r \\
& = & \sqrt{\frac{(q^n-(-1)^n)(q^{n-1}-(-1)^{n-1})}{q+1}} q^{c-n} \\
& \geq & \sqrt{\frac{q-1}{q+1}} \sqrt{\frac{(q^n-1)(q^{n-1}-1)}{q-1}} q^{c-n}.\end{eqnarray*} For $n \geq 3$ this is at least
\[ \sqrt{\frac{1}{3}} q^{n-1} q^{c-n} \geq .5 q^{c-1}.\]
 
To study the variance $\Var_{K^r}(f_C)$, we use Proposition \ref{combine} with $s=2$ and Theorem \ref{glutransv}. We conclude
that $\EV_{K^r}[f_C^2]$ is equal to $\frac{(q^n-(-1)^n)(q^{n-1}-(-1)^{n-1})}{q+1}$ multiplied by
\begin{eqnarray*}
& & \frac{q+1}{(q^n-(-1)^n)(q^{n-1}-(-1)^{n-1})} + \left( \frac{1}{q} \right)^r \frac{q^2-q-2}{(q^n-(-1)^n)(q^{n-1}-(-1)^{n-1})} \\
& & + \left( \frac{1}{q} \right)^{2r} \frac{q^{2n-1} - (-1)^{n-1}q^n - (-1)^n q^{n-1} -q^2}{(q^n-(-1)^n)(q^{n-1}-(-1)^{n-1})}.
\end{eqnarray*} Since
\[ (\EV_{K^r}(f_C))^2 = \frac{(q^n-(-1)^n)(q^{n-1}-(-1)^{n-1})}{q+1} \left( \frac{1}{q^2} \right)^r,\] we conclude that
\[ \Var_{K_r}[f_C] = 1 + T_1 + T_2 - T_3,\] where
\[\begin{aligned} T_1 & = \frac{1}{q^{n-c}} \frac{q^2-q-2}{q+1}, \\
T_2 & = \frac{1}{q^{2n-2c}} \frac{q^{2n-1} - (-1)^{n-1}q^n - (-1)^nq^{n-1}-q^2}{q+1}, \\
T_3 & = \frac{1}{q^{2n-2c}} \frac{(q^n-(-1)^n)(q^{n-1}-(-1)^{n-1})}{q+1}.\end{aligned}\] Now for $n \geq 2$,
\[ T_1 = \frac{q^c}{q^n} \frac{q^2-q-2}{q+1} \leq q^{c-1}. \] Also note that $T_2<T_3$. It follows
that $\Var_{K^r}(f_C) \leq 1+q^{c-1}$. Hence $K^r(A)$ is
\begin{eqnarray*}
K^r\left( f_C \leq \frac{1}{4} q^{c-1} \right) & = & K^r\left( f_C - \EV_{K^r}(f_C) \leq \frac{1}{4} q^{c-1} - \EV_{K^r}(f_C) \right) \\
& \leq & K^r\left( f_C - \EV_{K^r}(f_C) \leq \frac{1}{4} q^{c-1} - \frac{1}{2} q^{c-1} \right) \\
& = &   K^r\left( f_C - \EV_{K^r}(f_C) \leq -\frac{1}{4} q^{c-1} \right) \\
& \leq & K^r\left( |f_C-\EV_{K^r}(f_C)| \geq \frac{1}{4} q^{c-1} \right).
\end{eqnarray*} By Chebyshev's inequality this is at most
\[ \frac{16}{q^{2c-2}} \Var_{K^r}(f_C) \leq \frac{16(1+q^{c-1})}{q^{2c-2}}.\] The theorem
now follows for $r$ even since \[ ||K_{\triv}^r-\pi||_{\mathsf{TV}} \geq |K_{\triv}^r(A) - \pi(A)|.\]

Suppose next that $r$ is odd. Then arguing as in the $r$ even case,
\[ \EV_{K^r}[f(C)] = \sqrt{\frac{(q^n-(-1)^n)(q^{n-1}-(-1)^{n-1})}{q+1}} \frac{-1}{q^r} \leq -.5q^{c-1} \] for $n \geq 3$.
Moreover \[ \Var_{K^r}(f_C) = 1 - T_1 + T_2 - T_3 \] where $T_1,T_2,T_3$ are as in the r even case. So $\Var_{K^r}(f_C)$ in
the r odd case is at most $\Var_{K^r}(f_C)$ in the $r$ even case, which is at most $1+q^{c-1}$. Now use Chebyshev's inequality
exactly as in the r even case.
\end{proof}

\section{Weil representations of symplectic groups}

Section \ref{Spodd} treats Weil representations of odd characteristic symplectic groups. Section
\ref{Speven} treats Weil representations of even characteristic symplectic groups.

\subsection{Odd characteristic symplectic groups} \label{Spodd}

Recall by Theorem \ref{charbound} that for the Markov chain on irreducible representations of a finite group $G$
given by tensoring with the representation $\eta$,

\[ 4 ||K^r_{\triv} - \pi ||_{\mathsf{TV}}^2 \leq \sum_{C \neq \{\Id\}} |C|
\left| \frac{\chi^{\eta}(C)}{d_{\eta}} \right|^{2r} .\]

Suppose that $q$ is odd. A theorem of Howe \cite{Ho} shows that if $\eta$ is the
(reducible) Weil representation, then for any element $g$
of $\Spnq$, \[ |\chi^{\eta}(g)| = q^{\frac{1}{2} \dim(\ker(g-1))}.\] In
particular, $\chi^{\eta}(\Id)=q^n$. Thus from Theorem \ref{charbound},

\begin{eqnarray*}
4 ||K^r_{\triv} - \pi ||_{\mathsf{TV}}^2 & \leq & \sum_{k=0}^{2n-1} \sum_{g \atop \dim(\ker(g-1)=k)} \frac{q^{rk}}{q^{2nr}} \\
& = & \frac{|\Spnq|}{q^{2nr}} \sum_{k=0}^{2n-1} q^{rk} \PV_{\Sp}(2n,k),
\end{eqnarray*} where $\PV_{\Sp}(2n,k)$ is the probability that a random element of
$\Spnq$ has a $k$-dimensional fixed space.

The following theorem is from Rudvalis and Shinoda \cite{RS} and holds in both odd and even
characteristic.

\begin{theorem} \label{RudShi}
\begin{enumerate}[\rm(i)]
\item The proportion $\PV_{\Sp}(2n,2k)$ is equal to
\[ \frac{1}{|\Spkq|} \sum_{i=0}^{n-k} \frac{(-1)^i q^{i(i+1)}}{|\Spiq| q^{2ik}}.\]

\item The proportion $\PV_{\Sp}(2n,2k+1)$
is equal to \[ \frac{1}{|\Spkq| q^{2k+1}} \sum_{i=0}^{n-k-1} \frac{(-1)^i q^{i(i+1)}}
{|\Spiq| q^{2i(k+1)}}. \]
\end{enumerate}
\end{theorem}

Theorem \ref{oddcharmain} gives an upper bound on the convergence rate of our tensor product Markov chain.

\begin{theorem} \label{oddcharmain} Suppose that $q$ is odd, and let $K$ be the Markov chain on irreducible
representations of $\Spnq$ given by tensoring with the Weil representation. Then
for $r=2n+c$ with $c>0$,
\[ ||K_{\triv}^r-\pi||_{\mathsf{TV}} \leq \frac{1}{2 \sqrt{q^c-1}}.\]
\end{theorem}

\begin{proof} From the above discussion we conclude that an upper bound on $4 ||K^r_{\triv} - \pi ||_{\mathsf{TV}}^2$ is given by
\[ \frac{|\Spnq|}{q^{2nr}} \left[ \sum_{k=0}^{n-1} \PV_{\Sp}(2n,2k) q^{2rk} + \sum_{k=0}^{n-1}
\PV_{\Sp}(2n,2k+1) q^{r(2k+1)} \right]. \] From Theorem \ref{RudShi}, it is easily seen that
\[ \PV_{\Sp}(2n,2k) \leq \frac{1}{|\Spkq|} \ , \ \PV_{\Sp}(2n,2k+1) \leq \frac{1}{|\Spkq|q^{2k+1}}.\]

First let's upper bound \[ \frac{|\Spnq|}{q^{2nr}} \sum_{k=0}^{n-1} \PV_{\Sp}(2n,2k) q^{2rk} .\] Replacing $k$ by $n-1-k$, we need to upper bound
\[ \frac{|\Spnq|}{q^{2nr}} \sum_{k=0}^{n-1} \frac{q^{2r(n-1-k)}}{|Sp(2n-2-2k,q)|} =  \sum_{k=0}^{n-1} \frac{1}{q^{2r(k+1)}} \frac{|\Spnq|}{|Sp(2n-2-2k,q)|}.\] Since \[ |\Spnq| = q^{2{n^2+n}} (1-1/q^2) (1-1/q^4) \cdots (1-1/q^{2n}),\] this is at most
\begin{eqnarray*}
 \sum_{k=0}^{n-1} \frac{1}{q^{2r(k+1)}} \frac{q^{2n^2+n}}
{q^{2(n-k-1)^2+n-k-1}}
& = & \sum_{k=0}^{n-1} \frac{1}{q^{2r(k+1)} q^{-4n-4kn+3k+2k^2+1}} \\
& \leq & \sum_{k=0}^{n-1} \frac{1}{q^{2r(k+1)-4n-4kn}}.
\end{eqnarray*} Setting $r=2n+c$ with $c>0$, this is at most 
\[ \sum_{k=0}^{n-1} \frac{1}{q^{2c(k+1)}} \leq \frac{1}{q^{2c}-1}.\]

Next let's upper bound 
\[ \frac{|\Spnq|}{q^{2nr}} \sum_{k=0}^{n-1} \PV_{\Sp}(2n,2k+1) q^{r(2k+1)}
\leq \frac{|\Spnq|}{q^{2nr}} \sum_{k=0}^{n-1} \frac{q^{r(2k+1)}}{|\Spkq| q^{2k+1}}.\]
Replacing $k$ by $n-1-k$, we need to upper bound
\begin{eqnarray*}
& & \sum_{k=0}^{n-1} \frac{|\Spnq|}{|\Sp(2n-2-2k,q)|} \frac{1}{q^{r(2k+1)}} \frac{1}{q^{2n-2k-1}} \\
& \leq & \sum_{k=0}^{n-1} \frac{q^{2n^2+n}}{q^{2(n-k-1)^2+n-k-1}}
\frac{1}{q^{r(2k+1)}} \frac{1}{q^{2n-2k-1}} \\
& = & \sum_{k=0}^{n-1} \frac{1}{q^{-2n-4kn+2k^2+k+r(2k+1)}} \\
& \leq & \sum_{k=0}^{n-1} \frac{1}{q^{-2n-4kn+r(2k+1)}}.
\end{eqnarray*} Setting $r=2n+c$ with $c>0$, this is at most
\[ \sum_{k=0}^{n-1} \frac{1}{q^{c(2k+1)}} \leq \frac{1}{q^c(1-1/q^{2c})}.\]

Summarizing, we have shown that
\[ 4 ||K_{\triv}^r - \pi||_{\mathsf{TV}}^2 \leq \frac{1}{q^{2c}-1} + \frac{q^c}{q^{2c}-1} = \frac{1}{q^c-1}.\]
The result follows by taking square roots. \end{proof}

Next we prove a lower bound for the case $q=1$ mod 4. This is easier than the case $q=3$ mod 4 for two reasons.
According to \cite{FZ}, in this case every element in $\Spnq$ is real (i.e. conjugate to its inverse).
Hence if $C$ is one of the classes of transvections, then the function $f_C$ defined by \begin{equation}
\label{used}  f_C(\rho)= \frac{|C|^{1/2} \chi^{\rho}(C)}{d_{\rho}} \end{equation}
is real-valued. Second, $\chi^{\eta}(C)$ is real-valued (so equal to $\pm q^{(2n-1)/2}$ depending on which class of transvections you look at).
We also remark that the proof of Theorem \ref{lower1mod4} does not require $r$ to be even (so $c$ could be a half-integer).

\begin{theorem} \label{lower1mod4} Suppose that $q$ is congruent to 1 mod 4. Let $K_{\triv}^r$ be the distribution on $\Irr(\Spnq)$ given by $r$ steps of tensoring with the Weil representation $\eta$. Then given $\epsilon>0$, there exists $c>0$ (depending only on $\epsilon$) such that \[ ||K_{\triv}^r - \pi||_{\mathsf{TV}} \geq 1 - \epsilon\] for
$r=2n-2c$ and sufficiently large $n$.
\end{theorem}
 
\begin{proof} Let $C$ be the class of transvections satisfying $\chi^{\eta}(C)=q^{(2n-1)/2}$, and define $f_C$ by \eqref{used}. So $f_C$ is real
valued with mean 0 and variance 1 under Plancherel measure, and letting $A$ be the event that $f_C \leq \frac{q^c}{4}$, it follows from Chebyshev's
inequality that $\pi(A) \geq 1 - \frac{16}{q^{2c}}$.

From Proposition \ref{combine} with $s=1$, it follows that
\[ \EV_{K^r}(f_C) = |C|^{1/2} \left( \frac{\chi^{\eta}(C)}{d_{\eta}} \right)^r = \sqrt{\frac{q^{2n}-1}{2}} \frac{1}{q^{r/2}}.\] Since
$r=2n-2c$, this is equal to \[ \sqrt{\frac{q^{2n}-1}{2 q^{2n}}} q^c \geq \frac{1}{2} q^c.\] By Proposition \ref{combine} with $s=2$,
it follows that
\[ \EV_{K^r}[f_C^2] = |C| \sum_T p_2(T) \left[ \frac{\chi^{\eta(T)}}{d_{\eta}} \right]^r = \frac{q^{2n}-1}{2 q^{nr}} \sum_T p_2(T) \chi^{\eta}(T)^r,\]
where $p_2(T)$ is the chance that the product of two random elements of $C$ lies in $T$, and the sum is over all conjugacy classes $T$.
Theorem \ref{sp-odd-prob1} gives an exact formula for $\sum_T p_2(T) \chi^{\eta}(T)^r$, so we conclude that
\begin{eqnarray*}
\EV_{K^r}[f_C^2] & = & \frac{1}{2 q^{2r}} \left[ 2q^{2r} + q^{3r/2} \frac{q-1}{2} + (-1)^r q^{3r/2} \frac{q-5}{2} + q^r(q^{2n}-q)  \right] \\
& \leq & \frac{1}{2 q^{2r}} [2q^{2r} + q^{3r/2}(q-1) + q^r(q^{2n}-q)] \\
& = & 1 + \frac{1}{q^{r/2}} \frac{q-1}{2} + \frac{q^{2n}-q}{2 q^r}.
\end{eqnarray*} So \[ \Var_{K^r}(f_C^2) \leq 1 + T_1 + T_2 - T_3,\] where\[ T_1 =  \frac{1}{q^{r/2}} \frac{q-1}{2} \ , \ T_2 = \frac{q^{2n}-q}{2 q^r} \ , \ T_3 = \frac{q^{2n}-1}{2 q^r}.\] Since $T_1 = \frac{1}{q^{n-c}} \frac{q-1}{2}$ and $T_2<T_3$, it follows that
\[ \Var_{K^r}(f_C) \leq 1 + \frac{1}{2}(q-1) q^{c-n} \leq 1 + \frac{q^c}{2}.\]

Now
\begin{eqnarray*}
K^r \left( f_C \leq \frac{q^c}{4} \right) & = & K^r \left( f_C - \EV_{K^r}(f_C) \leq \frac{q^c}{4} - \EV_{K^r}(f_C) \right) \\
& \leq &  K^r \left( f_C - \EV_{K^r}(f_C) \leq \frac{q^c}{4} -\frac{q^c}{2} \right) \\
& = & K^r \left( f_C - \EV_{K^r}(f_C) \leq -\frac{q^c}{4} \right) \\
& \leq & K^r \left( |f_C - \EV_{K^r}(f(C))| \geq \frac{q^c}{4} \right).
\end{eqnarray*} By Chebyshev's inequality, this is at most
\[ \frac{\Var_{K^r}(f_C)}{(q^c/4)^2} \leq \frac{16(1+q^c/2)}{q^{2c}}.\]
This completes the proof since \[  ||K_{\triv}^r - \pi||_{\mathsf{TV}} \geq |K^r(A) - \pi(A)|.\]
\end{proof}

Next we treat the case that $q=3$ mod 4. This is more difficult than the $q=1$ mod 4 case since neither $\chi^{\rho}$ (where $\rho$ is an irreducible representation) nor $\chi^{\eta}$ need be real-valued when evaluated on a transvection. So we define a function $f_*$ on $\Irr(\Spnq)$ by \begin{equation} \label{realval} f_*(\rho)= \frac{|C_1|^{1/2} \chi^{\rho}(C_1)}{\sqrt{2} d_{\rho}} +  \frac{|C_2|^{1/2} \chi^{\rho}(C_2)}{\sqrt{2} d_{\rho}},\end{equation} where $C_1,C_2$ are the two conjugacy classes of transvections (each of size $\frac{q^{2n}-1}{2}$). Since $C_1^{-1}=C_2$, the function $f_*$ is real-valued. We need the following modification of Proposition \ref{combine}.

\begin{prop} \label{modify} Suppose that $q=3$ mod 4. Then \[ \EV_{K^r}[(f_*)^s] = (q^{2n}-1)^{s/2} \sum_T p_s(T) \left( \frac{\chi^{\eta}(T)}{d_{\eta}} \right)^r,\] where $p_s(T)$ is the probability that the product of s uniformly chosen transvections belongs to  the conjugacy class $T$, and $T$ ranges over all conjugacy classes of $\Spnq$. \end{prop}

\begin{proof} The proof is similar to that of Proposition \ref{combine}. 
\begin{eqnarray*}
\EV_{K^r}[(f_*)^s] & = & \sum_{\rho} K^r(\triv,\rho) f_*(\rho)^s \\
& = & \sum_{\rho} \frac{d_{\rho} m_{\rho}(\eta^r)}{d_{\eta}^r} f_*(\rho)^s \\
& = & \sum_{\rho} \frac{d_{\rho}}{|G|} \sum_g \overline{\chi^{\rho}(g)} \frac{\chi^{\eta}(g)^r}{d_{\eta}^r} f_*(\rho)^s.
\end{eqnarray*} Expanding $f_*(\rho)^s$ by the binomial theorem and using the fact that $|C_1|=|C_2|$, this becomes
\[ \frac{|C_1|^{s/2}}{2^{s/2}} \sum_T \left( \frac{\chi^{\eta}(T)}{d_{\eta}} \right)^r \sum_{a=0}^s {s \choose a}
\frac{|T|}{|G|} \sum_{\rho} d_{\rho}^2 \frac{\overline{\chi^{\rho}(T)}}{d_{\rho}}
\left( \frac{\chi^{\rho}(C_1)}{d_{\rho}}\right)^a \left( \frac{\chi^{\rho}(C_2)}{d_{\rho}}\right)^{s-a}.\] 

By Fourier analysis
(see page 471 of \cite{Stan}), this is equal to
\[ 2^{s/2} |C_1|^{s/2} \sum_T \left[ \frac{\chi^{\eta}(T)}{d_{\eta}} \right]^r \sum_{a=0}^s \frac{{s \choose a}}{2^s}
p_{s,a}(T), \] where $p_{s,a}(T)$ is the probability that the product of $a$ elements of $C_1$ with $s-a$ elements of $C_2$
(in any fixed order) belongs to the conjugacy class $T$. Clearly
\[ \sum_{a=0}^s \frac{{s \choose a}}{2^s} p_{s,a}(T) = p_s(T),\] so the result follows since $2 |C_1| = q^{2n}-1$.
 \end{proof}
 
Now we prove the desired lower bound. We remark that the proof does require $r$ to be even.

\begin{theorem} \label{lower3mod4} Suppose that $q$ is congruent to 3 mod 4 and that $r$ is even. Let $K_{\triv}^r$ be the distribution on $\Irr(\Spnq)$ given by $r$ steps of tensoring with the Weil representation $\eta$. Then given $\epsilon>0$, there exists $c>0$ (depending only on $\epsilon$) such that \[ ||K_{\triv}^r - \pi||_{\mathsf{TV}} \geq 1 - \epsilon\] for $r=2n-2c$ and sufficiently large $n$.
\end{theorem}

\begin{proof} Define $f_*$ as in \eqref{realval}. So $f_*$ is real-valued, and it is easy to see from the orthogonality relations for irreducible characters
that if $\rho$ is chosen from Plancherel measure, then $\EV_{\pi}(f_*)=0$ and $\Var_{\pi}(f_*)=1$. So by Chebyshev's inequality,
\[ \pi \left(f_* \leq \frac{1}{2} q^c \right) \geq 1 - \frac{4}{q^{2c}}.\] 

We now consider two cases. Case 1 is that $r=0$ mod 4. By Proposition \ref{modify} with $s=1$, we have that
\begin{eqnarray*}
\EV_{K^r}(f_*) & = & (q^{2n}-1)^{1/2} \frac{1}{q^{nr}} \sum_T [\chi^{\eta}(T)]^r p_1(T) \\
& = &  (q^{2n}-1)^{1/2} \frac{1}{q^{nr}} \frac{1}{2} [\chi^{\eta}(C_1)^r + \chi^{\eta}(C_2)^r].
\end{eqnarray*} Now by \eqref{eq:weil2}, $\chi^{\eta}(C_1)$ and $\chi^{\eta}(C_2)$ are equal to
$iq^{(2n-1)/2}$ and $-i q^{(2n-1)/2}$, so since $r=0$ mod 4,
\[ \EV_{K^r}(f_*) = \sqrt{q^{2n}-1} \frac{1}{q^{r/2}} = \sqrt{q^{2n}-1} \frac{1}{q^{n-c}} \geq .8 q^c.\] 

By Proposition \ref{modify} with $s=2$ and Theorem \ref{sp-odd-prob2}, we have that 
\begin{eqnarray*}
\EV_{K^r}(f_*^2) & = & (q^{2n}-1) \frac{1}{q^{nr}} \sum_T [\chi^{\eta}(T)]^r p_2(T) \\
& = & \frac{(q^{2n}-1)}{q^{nr}} \frac{q^{(n-2)r}}{(q^{2n}-1)} (q^{2r} + (q-2)q^{3r/2} + q^r(q^{2n}-q)) \\
& = & 1 + \frac{(q-2)}{q^{r/2}} + \frac{(q^{2n}-q)}{q^r}.
\end{eqnarray*}

So $\Var_{K^r}(f) \leq 1 + T_1 + T_2 - T_3,$ where
\[T_1 = \frac{q-2}{q^{r/2}} \ , \ T_2 = \frac{q^{2n}-q}{q^r} \ , \ T_3 = \frac{q^{2n}-1}{q^r}.\]
Now $T_1 = \frac{q-2}{q^{n-c}} \leq q^c$ and $T_2<T_3$, so $\Var_{K^r}(f_*) \leq 1+q^c$.

Now
\begin{eqnarray*}
K^r \left(f_* \leq \frac{1}{2}q^c \right) & = & K^r \left( f_*-\EV_{K^r}(f_*) \leq \frac{1}{2}q^c - \EV_{K^r}(f_*) \right) \\
& \leq & K^r \left(f_* - \EV_{K^r}(f_*) \leq \frac{1}{2}q^c - .8q^c \right) \\
& = & K^r(f_*-\EV_{K^r}(f_*) \leq -.3q^c) \\
& \leq & K^r(|f_*-\EV_{K^r}(f_*)| \geq .3q^c).
\end{eqnarray*} By Chebyshev's inequality, this is at most
\[ \frac{\Var_{K^r}(f_*)}{.3^2 q^{2c}} \leq \frac{12 (1+q^c)}{q^{2c}}.\]
The result follows for Case 1 since
\[ ||K_{\triv}^r-\pi||_{\mathsf{TV}} \geq |K^r(A)-\pi(A)| \] for all $A$, and in particular for $A$ the event that $f_* \leq \frac{1}{2}q^c$.

Case 2 is that $r=2$ mod 4 and is very similar. First note that if $\rho$ is chosen from Plancherel measure, then
Chebyshev's inequality gives that \[ \pi \left(f_* \geq - \frac{1}{2} q^c \right) \geq 1 - \frac{4}{q^{2c}}.\] 

Arguing as in the $r=0$ mod 4 case, Proposition \ref{modify} with $s=1$ implies that  
\[ \EV_{K^r}(f_*) = \sqrt{q^{2n}-1} \frac{-1}{q^{r/2}} \leq -.8q^c.\]

By Proposition \ref{modify} with $s=2$ and Theorem \ref{sp-odd-prob2}, we have that
\begin{eqnarray*}
\EV_{K^r}(f_*^2) & = & \frac{(q^{2n}-1)}{q^{nr}} \sum_T [\chi^{\eta}(T)]^r p_2(T) \\
& = & \frac{(q^{2n}-1)}{q^{nr}} \frac{q^{(n-2)r}}{q^{2n}-1} (q^{2r}-(q-2)q^{3r/2}+q^r(q^{2n}-q)) \\
& \leq & 1 + \frac{q^{2n}-q}{q^r}.
\end{eqnarray*} So \[ \Var_{K^r}(f_*) \leq 1 + \frac{q^{2n}-q}{q^r} - \frac{q^{2n}-1}{q^r} \leq 1.\]

Now
\begin{eqnarray*}
K^r \left( f_* \geq \frac{-1}{2}q^c \right) & = & K^r \left( f_*-\EV_{K^r}(f_*) \geq \frac{-1}{2}q^c - \EV_{K^r}(f_*)  \right) \\
& \leq & K^r \left(   f_*-\EV_{K^r}(f_*) \geq \frac{-1}{2}q^c + .8q^c  \right) \\
& = & K^r (f_*-\EV_{K^r}(f_*) \geq .3q^c) \\
& \leq & K^r(|f_*-\EV_{K^r}(f_*)| \geq .3q^c). \end{eqnarray*} By Chebyshev's inequality this is at most
$$\frac{\Var_{K^r}(f_*)}{.3^2 q^{2c}} \leq \frac{12}{q^{2c}},$$ 
completing the proof as in the
$r=0$ mod 4 case. \end{proof}

\subsection{Even characteristic symplectic groups} \label{Speven}

For even characteristic symplectic groups there are two reducible Weil characters: {\it linear Weil character}
\[ \tau(g) = q^{\dim(\ker(g-1))} \] and {\it unitary Weil character}
\[ \zeta(g) = (-q)^{\dim(\ker(g-1))}, \]
see \cite[\S3]{GT}.

The upper bounds obtained by Theorem \ref{charbound} are the same for both of these characters due to the even exponent $2r$ in the
statement of Theorem \ref{charbound}.

\begin{theorem} \label{evenupbound} Suppose that $q$ is even, and let $K$ be the Markov chain on irreducible representations of $\Spnq$
given by tensoring with the Weil representation $\tau$. Then for $r=n+c$ with $c>0$, 
\[ ||K^r_{\triv} - \pi||_{\mathsf{TV}} \leq \frac{1}{2 \sqrt{q^{2c}-1}}.\] The same is true for the representation $\zeta$.
\end{theorem}

\begin{proof} Arguing as in the odd characteristic case, the upper bound on $4 ||K^r_{\triv} - \pi||_{\mathsf{TV}}^2$ from Theorem \ref{charbound}
in even characteristic is equal to
\[ \frac{|\Spnq|}{q^{4nr}} \sum_{k=0}^{2n-1} q^{2rk} \PV_{\Sp}(2n,k) \] where $\PV_{\Sp}(2n,k)$ is the probability that a random element of
$\Spnq$ has a k-dimensional fixed space. Since $\PV_{\Sp}(2n,k)$ is the same in odd and even characteristic \cite{RS}, this is exactly equal to the upper bound we used for \[ 4 ||K^{2r}_{\triv} - \pi||_{\mathsf{TV}}^2 \] in odd characteristic. So the result is immediate from Theorem \ref{oddcharmain}.
\end{proof}

To prove a lower bound, we need a combinatorial result about transvections (Theorem \ref{sp-transv}). Now we prove the following lower bound. 

\begin{theorem}\label{lower-Sp-even} Suppose that $q$ is even. Let $K_{\triv}^r$ be the distribution on $\Irr(\Spnq)$ given by $r$ steps of tensoring with
the Weil representation $\tau$. Then given $\epsilon>0$, there exists $c>0$ (depending only on $\epsilon$) such that \[ ||K_{\triv}^r - \pi||_{\mathsf{TV}} \geq 1 - \epsilon\] for
$r=n-c$ and sufficiently large $n$. The same holds for the Weil representation $\zeta$.
\end{theorem}

\begin{proof} Let us first work with $\tau$. Let $C$ be the conjugacy class of transvections in $\Spnq$ (since $q$ is even there is only
one class of transvections). Define the function $f_C$ on $\Irr(\Spnq)$ by \[ f_C(\rho) = \frac{|C|^{1/2} \chi^{\rho}(C)}{d_{\rho}}.\] The
orthogonality relations for irreducible characters imply that under the Plancherel measure $\pi$, the random variable $f_C$ has mean 0 and variance 1.
So letting $A$ be the event $f_C \leq \alpha$ for $\alpha>0$, Chebyschev's inequality gives that $\pi(A) \leq 1 - \frac{1}{\alpha^2}$. Letting
$\alpha = \frac{q^c}{2}$, it follows that \[ \pi(A) \geq 1 - \frac{4}{q^{2c}}.\] 

Next we will lower bound the mean $\EV_{K^r}(f_C)$ and upper bound the variance $\Var_{K^r}(f_C)$. By Proposition \ref{combine},
\[ \EV_{K^r}(f_C) = |C|^{1/2} \left( \frac{\chi^{\tau}(C)}{d_{\tau}} \right)^r = \sqrt{q^{2n}-1} \cdot q^{c-n} \geq \sqrt{\frac{3 q^{2n}}{4}} q^{c-n} 
\geq .8 q^c.\]

To study the variance $\Var_{K^r}(f_C)$, we use Proposition \ref{combine} with $s=2$ and Theorem \ref{sp-transv} to conclude that
$\EV_{K^r}[(f_C)^2]$ is equal to 
 \[ (q^{2n}-1) \left[ \frac{1}{q^{2n}-1} + \frac{q-2}{q^{2n}-1} \frac{1}{q^r} + \frac{q^{2n}-q}{q^{2n}-1} \frac{1}{q^{2r}} \right].\]
Since \[ (\EV_{K^r}(f_C))^2 = \frac{q^{2n}-1}{q^{2n-2c}},\] it follows that $\Var_{K^r}(f_C)$ is equal to $1 +T_1 + T_2 - T_3$, where
\begin{equation} \label{3ts}  T_1 = \frac{q-2}{q^{n-c}} \ , \ T_2 = \frac{q^{2n}-q}{q^{2n-2c}} \ , \ T_3 = \frac{q^{2n}-1}{q^{2n-2c}}.
\end{equation} Now if $n \geq 2$,
then $T_1 \leq \frac{q^c}{2}$ and $T_2<T_3$. So for $n \geq 2$, $\Var_{K^r}(f_C) \leq 1 + \frac{q^c}{2}$.

Now \begin{eqnarray*}
K^r\left(f_C \leq \frac{1}{2} q^c \right) & = & K^r\left(f_C - \EV_{K^r}(f_C) \leq \frac{1}{2} q^c - \EV_{K^r}(f_C) \right) \\
& \leq & K^r \left(f_C -\EV_{K^r}(f_C) \leq \frac{1}{2} q^c - .8 q^c \right) \\
& = & K^r \left(f_C -\EV_{K^r}(f_C) \leq -.3 q^c \right) \\
& \leq & K^r \left(|f_C -\EV_{K^r}(f_C)| \geq .3 q^c \right).
\end{eqnarray*} By Chebyshev's inequality, this is at most
\[ \frac{\Var_{K^r}(f_C)}{(.3 q^c)^2} \leq \frac{12(1+ \frac{q^c}{2})}{q^{2c}}.\]

Since \[ ||K^r - \pi||_{\mathsf{TV}} \geq |K^r(A) - \pi(A)|,\] the theorem is proved for the Weil representation $\tau$.

For the Weil representation $\zeta$, we break into cases based on the parity of $r$. If $r$ is even then arguing as for $\tau$, one has
that \[ \EV_{K^r}[f_C] = |C|^{1/2} \left[ \frac{\chi^{\zeta}(C)}{d_{\zeta}} \right]^r \geq .8q^c.\] Moreover for $r$ even,  $\EV_{K^r}[(f_C)^2]$ is
exactly the same as for $\tau$, so the theorem is proved for $r$ even. If $r$ is odd, then $\EV_{K^r}(f_C) \leq -.8 q^c$. Moreover
\[ \Var_{K^r}(f_C) = 1 - T_1 + T_2 - T_3, \] where $T_1,T_2,T_3$ are defined in \eqref{3ts}.  So the variance of $f_C$ under the
measure $K^r$ is smaller for $\zeta$ than for $\tau$, so the lower bound for $\zeta$ is all the more true than the lower bound for $\tau$.
\end{proof}

\section{Products of two transvections in $\GLnq$}\label{sec:gl1}

Let  $F$ be a field and $V$ a finite dimensional $F$-vector space.  A linear transformation $T\in \GL(V)$ is a transvection if and only if $T-I$ is nilpotent and of rank $1$,
i.e., if and only if $T$ is of the form 
$$I+ vw^*: y \mapsto y +w^*(y)v,$$ 
where $v\in V$ and $w^*\in V^*$ are non-zero vectors such that $w^*(v)=0$.
If $|F|=q$ is finite, then the map $\tau(v,w^*)\mapsto I+vw^*$ from 
$$X_V := \{(v,w^*)\in V\times V^*\mid v\neq 0,w^*\neq 0,w^*(v)=0\}$$ 
to the set of transvections in $\GL(V)$ is $q-1$ to $1$.

From now on, we assume $F=\F_q$ and $\dim_{\F_q}V=n \geq 2$. Then 
$$|X_V| = (q^n-1)(q^{n-1}-1)$$ 
since for any non-zero $w^*\in V^*$, the $v$ such that $(v,w^*)$ belongs to $X_V$ range over all non-zero vectors of $\ker w^*$.

Given $x_i = (v_i,w_i^*)\in X_V$ for $i=1,2$, we define $T_i = \tau(x_i)$ and consider the kernel of
$$N_{12} := T_1 T_2 - I = v_1 w_1^* + (v_2 + w_1^*(v_2)v_1) w_2^*.$$
For each $x_1\in X_V$ there are $q-1$ elements $x_2\in X_V$ such that $\tau(x_2) = \tau(x_1)^{-1}$, so the number of pairs $(x_1,x_2)\in X_V^2$ such that $\ker N_{12} = V$ is $(q-1)|X_V|$.

Since $\ker N_{12}$ contains $\ker w_1^* \cap \ker w_2^*$, which is of codimension $\le 2$ in $V$, our goal is to count the number of pairs $(x_1,x_2)$ such that $\dim\ker N_{12} = n-2$.
A necessary condition for this to happen is that $\ker w_1^*\neq \ker w_2^*$, or (equivalently) that $w_1^*$ and $w_2^*$ are linearly independent.  Assuming this is so, 
$\dim\ker N_{12} = n-2$ if and only if $v_1$ and $v_2 + w_1^*(v_2)v_1$ are linearly independent, which is equivalent to $v_1$ and $v_2$ being linearly independent.

\begin{lemma}\label{count1}
The number of pairs $(x_1,x_2)\in X_V^2$ such that $v_1$ and $v_2$ are linearly independent as are $w_1^*$ and $w_2^*$ is 
$$(q^{2n-1}-3q^n+q^{n-1}+q^2)|X_V|.$$
\end{lemma}

\begin{proof}
First we count pairs $(x_1,x_2)$ where $v_1$ and $v_2$ are dependent as are $w_1^*$ and $w_2^*$.
Given $x_1$, there are $q-1$ possibilities for $v_2$ and $q-1$ possibilities for $w_2^*$, for a total of $(q-1)^2|X_V|$.
Next we count pairs $(x_1,x_2)$ where $v_1$ and $v_2$ are dependent, but we make no assumption about $w_1^*$ and $w_2^*$.
Given $v_1$, there are $q-1$ possibilities for $v_2$ and $q^{n-1}-1$ possibilities each for $w_1^*$ and $w_2^*$, for a total of 
$$(q-1)(q^{n-1}-1)^2(q^n-1) = (q-1)(q^{n-1}-1)|X_V|.$$
The same thing applies for the condition that $w_1^*$ and $w_2^*$ are dependent but we make no assumption about $v_1$ and $v_2$.
Therefore, the total number of pairs $(x_1,x_2)$ where both the $v$-coordinates and the $w^*$-coordinates are independent is
$$(|X_V|-2(q-1)(q^{n-1}-1) + (q-1)^2)|X_V|,$$
which gives the lemma.
\end{proof}

Lemma \ref{count1} immediately implies the following result 

\begin{theorem}\label{gl-transv}
Let $V = \F_q^n$ with $n \geq 2$. Then the probability that the product of two uniformly distributed transvections in $\GL(V)$ has fixed point subspace of codimension $e$ is 
$$\frac{q^{2n-1}-3q^n+q^{n-1}+q^2}{(q^n-1)(q^{n-1}-1)}$$
if $e=2$, 
$$\frac{2q^{n}-2q^{n-1}-q^2-q+2}{(q^n-1)(q^{n-1}-1)}$$
if $e=1$, and
%
%
$$\frac{q-1}{(q^n-1)(q^{n-1}-1)}$$
if $e=0$.
\end{theorem}

\section{Products of two transvections in $\GUnq$}\label{sec:gu1}

Let $V$ be a vector space of dimension $n \geq 2$ over $\F_{q^2}$ endowed with a non-degenerate Hermitian inner product $(\cdot|\cdot)$, which we assume to be
$\F_{q^2}$-linear on the second argument. This inner product can be thought of as giving a semi-linear isomorphism $V\to V^*$, and for $w^*\in V^*$, we denote by $w$ the corresponding element of $V$: $w^*(y) = (w|y)$. As mentioned in \S\ref{sec:gl1}, the transformation $I + v w^*$ on $V$ is a transvection if 
and only if $v,w$ are nonzero, and 
\begin{equation}\label{eq:gu20}
  0=w^*(v) = (w|v). 
\end{equation}
We now show that it preserves the form $(\cdot|\cdot)$ if and only if 
\begin{equation}\label{eq:gu22}
  v \mbox{ is isotropic and }w=c v\mbox{ for some }c\in \F_{q^2}^\times \mbox{ with }c+c^q=0.
\end{equation} 
Indeed, this happens precisely when for any $x,y \in V$ we have
$$\begin{aligned}0 & =(x+(w,x)v|y+(w,y)v)-(x|y)\\
  & = (x|v)(w|y)+(w|x)^q(v|y)+(w|x)^q(w|y)(v|v)\\
  & =  \bigl( (v|x)w+(w|x)v+(w|x)(v|v)w|y \bigr).\end{aligned}.$$
Since $(\cdot|\cdot)$ is non-degenerate, this is equivalent to 
\begin{equation}\label{eq:gu21}
  (v|x)w+(w|x)v+(w|x)(v|v)w=0 
\end{equation}
for all $x \in V$. If $v$ and $w$ are linearly independent, then we
get $(w|x)=0$ for all $x \in V$ and hence $w=0$, a contradiction. So we can write $w=cv$ with $c \in \F_{q^2}^\times$, and \eqref{eq:gu20} and \eqref{eq:gu21} mean exactly that $v$ is isotropic and $c(v|x)v+c^q(v|x)v=0$. The latter can hold for all $x \in V$ precisely when $c+c^q=0$, as we claimed.

The number of isotropic vectors in $V$ is 
$$N := (q^n-(-1)^n)(q^{n-1}-(-1)^{n-1}).$$  
The number of pairs $(w^*,v)$ satisfying \eqref{eq:gu22} is therefore $(q-1)N$.  The map $\tau$ from pairs $(w^*,v)$ to unitary transvections is $q^2-1$ to $1$.  This is because pairs $(w_1^*,v_1)$ and $(w_2^*,v_2)$ correspond to the same unitary transvection if and only if $v_2 = \alpha v_1$ and $w_2 = \alpha^{-q} w_1$ for some $\alpha\in \F_{q^2}^\times$.  We conclude that the number of unitary transvections is $N/(q+1)$, and the probability that the product of two random symplectic transvections has fixed point space of codimension $0$ is accordingly $(q+1)/N$.

The product of $\tau(w_1^*,v_1)$ and $\tau(w_2^*,v_2)$ is of rank $2$ if and only if $v_1$ and $v_2$ are linearly independent, since, as in the symplectic case, this implies also that $w_1^*$ and $w_2^*$ are linearly independent as well.  The number of pairs $(v_1,v_2)$ of isotropic vectors which are linearly independent is $N(N-(q^2-1))$.  For each such pair there are $(q-1)^2$ possibilities for $(w_1^*,w_2^*)$, so the probability that a random pair $(T_1,T_2)$ of unitary transvections has $I-T_1 T_2$ of rank $2$ is
$$\frac{N-q^2+1}{N} = \frac{q^{2n-1} - (-1)^{n-1}q^n - (-1)^n q^{n-1} -q^2}{(q^n-(-1)^n)(q^{n-1}-(-1)^{n-1})}.$$
We have proved
\begin{theorem}\label{glutransv}
Let $V = \F_{q^2}^n$ with $n \geq 2$. Then the probability that the product of two uniformly distributed transvections in $\GU(V) \cong \GUnq$ has fixed point subspace of codimension $e$ is 
$$\frac{q^{2n-1} - (-1)^{n-1}q^n - (-1)^n q^{n-1} -q^2}{(q^n-(-1)^n)(q^{n-1}-(-1)^{n-1})}$$
if $e=2$, 
$$\frac{q^2-q-2}{(q^n-(-1)^n)(q^{n-1}-(-1)^{n-1})}$$
if $e=1$, and
%
%
$$\frac{q+1}{(q^n-(-1)^n)(q^{n-1}-(-1)^{n-1})}$$
if $e=0$.
\end{theorem}

\section{Products of two transvections in $\Spnq$}\label{sec:sp1}

Now assume that $V=\F_q^{2n}$ is endowed with a non-degenerate symplectic form $(\cdot|\cdot)$.  In particular, this form allows us to identify $V$ and $V^*$, and for $w^*\in V^*$, we denote by $w$ the corresponding element of $V$: $w^*(y) = (w|y)$.
The transvection $I + vw^* $ respects $(\cdot,\cdot)$ if and only if $w$ and $v$ are scalar multiples of one another.  Let $Y_V\subset X_V$ denote the set of ordered pairs $(v,w^*)\in V\times V^*$ such that
$v$ and $w$ are non-zero vectors which are multiples of one another.  Thus, 
$$|Y_V| = (q^{2n}-1)(q-1).$$  The restriction of $\tau$ to $Y_V$ is a $q-1$ to $1$ map from $Y_V$ to the set of symplectic transvections of $\Sp(V)$.
It follows that there are $q^{2n}-1$ symplectic transvections, and the probability that the product of two random symplectic transvections has fixed point space of codimension $0$ is accordingly
$$\frac {1}{q^{2n}-1}.$$

We now calculate the probability that the codimension is $2$.  As in the $\GL(V)$ case, this happens if and only if $w_1^*$ and $w_2^*$ are linearly independent and $v_1$ and $v_2$ are linearly independent.
However, the two conditions are equivalent, so we need only count the number of pairs $(x_1,x_2)$ such that $v_1$ and $v_2$ are linearly independent.
Given any ordered pair $(v_1,v_2)$ of linearly independent vectors, there are $(q-1)^2$ choices for $(w_1^*,w_2^*)$, so the total number of possibilities is 
$$(q-1)^2 (q^{2n}-q)(q^{2n}-1) = (q-1)(q^{2n}-q)|Y_V|.$$
Therefore, the probability that the invariant subspace of the product of two random symplectic transvections will have codimension $2$ is
$$\frac{q^{2n}-q}{q^{2n}-1}.$$
We have proved 

\begin{theorem}\label{sp-transv}
Let $V = \F_q^{2n}$ with $n \geq 1$ be a symplectic space. Then the probability that the product of two uniformly distributed transvections in $\Sp(V)$ has fixed point subspace of codimension $e$ is 
$$\frac{q^{2n}-q}{q^{2n}-1}$$
if $e=2$, 
$$\frac{q-2}{q^{2n}-1}$$
if $e=1$, and  
$$\frac{1}{q^{2n}-1}$$
if $e=0$.
\end{theorem}

From now on, we assume $q$ is an odd prime power, and let $\F_q^{\times 2}$ denote the set of squares in $\F_q^\times$. We will frequently identify the integer $1$ with 
the identity in $\F_q$, and set
$$\kappa := (-1)^{(q-1)/2}.$$

\begin{lemma}\label{lem:sq1}
The equation $x+1=y$ has $(q-5)/4$ solutions $(x,y) \in \F_q^{\times 2} \times \F_q^{\times 2}$ if $q \equiv 1 \pmod{4}$, and $(q-3)/4$ solutions 
$(x,y) \in \F_q^{\times 2} \times \F_q^{\times 2}$ if $q \equiv 3 \pmod{4}$.
\end{lemma}

\begin{proof}
Write $x=u^2$ and $y=v^2$, so that $u^2-v^2=-1$, and we want $u,v \in \F_q^\times$. With $z:=u+v$ we have $u-v = -1/z$, and hence
$u=(z-1/z)/2$ and $v=(z+1/z)/2$, where $z \in \F_q^\times$. If $q \equiv 1 \pmod{4}$, then $u,v \neq 0$ precisely when $z^2 \neq \pm 1$, leaving $q-5$
choices for $z$, $(q-5)/2$ choices for $v$, and hence $(q-5)/4$ choices for $y$.  If $q \equiv 3 \pmod{4}$, then $u,v \neq 0$ precisely when $z^2 \neq 1$, leaving $q-3$
choices for $z$, $(q-3)/2$ choices for $v$, and hence $(q-3)/4$ choices for $y$.  
\end{proof}

Following the notation of \cite{Sr}, we fix an element $\theta \in \F_{q^2}^\times$ of order $q^2-1$, and let 
$$\gamma:=\theta^{q+1},~\eta:=\theta^{q-1}.$$

\begin{lemma}\label{lem:sq2}
\begin{enumerate}[\rm(i)]
\item Suppose $q \equiv 1 \pmod{4}$. When $\alpha$ runs over the set $\F_q^{\times 2} \setminus \{4\}$, $2-\alpha$ can be written as 
$\lambda+\lambda^{-1}$ with $\lambda = \gamma^i$ for exactly $(q-5)/4$ values of $\alpha$, for all of which we have $2 \mid i$, and  
$\lambda+\lambda^{-1}$ with $\lambda = \eta^j$ for exactly $(q-1)/4$ values of $\alpha$, for all of which we have $2 \nmid j$.
\item Suppose $q \equiv 3 \pmod{4}$. When $\alpha$ runs over the set $\F_q^{\times 2} \setminus \{4\}$, $2-\alpha$ can be written as 
$\lambda+\lambda^{-1}$ with $\lambda = \gamma^i \neq \pm 1$ for exactly $(q-3)/4$ values of $\alpha$, for all of which we have $2 \nmid i$, and  
$\lambda+\lambda^{-1}$ with $\lambda = \eta^j \neq \pm 1$ for exactly $(q-3)/4$ values of $\alpha$, for all of which we have $2 \mid j$.
\end{enumerate}
\end{lemma}

\begin{proof}
Consider the element $X=\begin{pmatrix}1-\alpha & 1\\-\alpha & 1\end{pmatrix} \in \SLtwoq$ with trace $2-\alpha$. It is clear that  neither $1$ nor $-1$ 
is not an eigenvalue of $X$. Hence $X$ has two distinct eigenvalues $\lambda$ and $\lambda^{-1}$, whence 
\begin{equation}\label{eq:sq20}
  2-\alpha = \lambda+\lambda^{-1}.
\end{equation}  
Since $X \in \SLtwoq$, we have $\lambda^q = \lambda$ or $\lambda^q=-\lambda^{-1}$, and the total number of these two occurrences is $(q-3)/2$.

In the first case, we have $\lambda=\gamma^i \neq \pm 1$. Furthermore, \eqref{eq:sq20} implies $(\lambda-1)^2 = -\lambda\alpha$, showing 
that $-\gamma^i=-\lambda \in \F_q^{\times 2}$. If $q \equiv 1 \pmod4$, then $\pm 1 \neq \gamma^i \in \F_q^{\times 2}$, showing
$2|i$ and leaving $(q-5)/2$ possibilities for $\lambda$ and $(q-5)/4$ possibilities for $\alpha$. If $q \equiv 3 \pmod4$, then $-1 \neq \gamma^i$ is a non-square in 
$\F_q^\times$, showing $2 \nmid i$ and leaving $(q-3)/2$ possibilities for $\lambda$ and $(q-3)/4$ possibilities for $\alpha$ in this case.

In the second case, $\lambda=\eta^j \neq \pm 1$ and $\lambda^{(q+1)/2} = (-1)^j$. As before, $(\lambda-1)^2/\lambda = -\alpha$, whence 
$$\begin{aligned}\kappa  = (-\alpha)^{(q-1)/2} & = \frac{(\lambda-1)^{q-1}}{\lambda^{(q-1)/2}}=  \frac{(\lambda-1)^{q}}{\lambda^{(q-1)/2}(\lambda-1)}
    =  \frac{\lambda^q-1}{\lambda^{(q-1)/2}(\lambda-1)}\\
    &  = \frac{\lambda^{-1}-1}{\lambda^{(q-1)/2}(\lambda-1)} = \frac{-1}{\lambda^{(q+1)/2}}  = \frac{-1}{(-1)^j},\end{aligned}$$
and thus $(-1)^j = -\kappa$. Hence $2 \nmid j$ if $q \equiv 1 \pmod4$, and $2 \mid j$ if $q \equiv 3 \pmod4$.  
\end{proof}

Unlike $\GLnq$, with $2 \nmid q$ symplectic transvections in $\Spnq$ break into two conjugacy classes, say $C$ and $C^*$. We will work with
one of them, say $C$, which may be assumed to
be consisting of 
$$T(\alpha,v): x  \mapsto x + \alpha(v|x)v,$$
where $0 \neq v \in V$ and $\alpha \in \F_q^{\times 2}$. As before, the map $(\alpha,v) \mapsto T(\alpha,v)$ is $(q-1)$ to $1$, and 
$$|C|=(q^{2n}-1)/2.$$

\smallskip
We will now determine the probability that the product $S:=T(\alpha,u)T(\beta,v)$ of two random transvections from $C$ belongs to a fixed conjugacy class in $\Spnq$. 

\smallskip
First we consider the case $S$ has rank $\leq 1$. As mentioned above, this can only happen when $u$ and $v$ are linearly dependent. Rescaling $u$, which leads to 
a factor $q-1$ in counting the number of possibilities, we may assume $u=v$. Then $S=T(\alpha+\beta,v)$. In particular, $S$ can be the identity exactly when $\beta=-\alpha$ for 
$\alpha,\beta \in \F_q^{\times 2}$. Hence the probability that $S$ is the identity is $1/|C|$ if $q \equiv 1 \pmod4$ and $0$ otherwise. 

Suppose $q \equiv 1 \pmod 4$. By Lemma \ref{lem:sq1}, among the $(q-1)^2/4$ pairs $(\alpha,\beta) \in \F_q^{\times 2} \times \F_q^{\times 2}$,  $(q-1)(q-5)/8$
of them have $\alpha+\beta \in \F_q^{\times 2}$, and hence $(q-1)^2/8$
of them have $\alpha+\beta$ being a non-square in $\F_q^{\times}$. Hence the probability that $S$ belongs to $C$,  respectively $C^*$, is 
$$\frac{(q-5)/2}{q^{2n}-1}, \mbox{ resp. }\frac{(q-1)/2}{q^{2n}-1}.$$ 

Suppose $q \equiv 3 \pmod 4$. By Lemma \ref{lem:sq1}, among the $(q-1)^2/4$ pairs $(\alpha,\beta) \in \F_q^{\times 2} \times \F_q^{\times 2}$,  $(q-1)(q-3)/8$
of them have $\alpha+\beta \in \F_q^{\times 2}$, and hence $(q-1)(q+1)/8$
of them have $\alpha+\beta$ being a non-square in $\F_q^{\times}$. Hence the probability that $S$ belongs to $C$,  respectively $C^*$, is 
$$\frac{(q-3)/2}{q^{2n}-1}, \mbox{ resp. }\frac{(q+1)/2}{q^{2n}-1}.$$ 

Now we consider the case $S$ has rank $2$, i. e. $u$ and $v$ are linearly independent. Conjugating inside $\Spnq$, we may assume that 
$u,v \in V_1 = \F_q^4$, a non-degenerate subspace of $V = V_1 \oplus V_1^{\perp}$. It follows that $S$ acts trivially on $V_1^\perp$ and stabilizes $V_1$, and 
hence its conjugacy class is completely determined by the conjugacy class in $\Sp(V_1) \cong \Spfourq$ that contains $S|_{V_1}$. For the latter classes, we will follow
the notation in \cite{Sr}, and in fact will use that notation to denote the class of $S$ in $\Spnq$. In particular, $C$ can be denoted by $A_{21}$, and $C^*$ by
$A_{22}$. Another advantage of this labeling is that, since $\Sp(V_1)$ is a standard subgroup of $\Sp(V)$, we can choose a (reducible) Weil representation 
$$\omega=\omega_n$$ 
of degree $q^n$ of $\Spnq$ such that
\begin{equation}\label{eq:weil1}
  \omega(S) = q^{n-2}\omega_4(S|_{V_1}) = q^{n-2}(\theta_3-\theta_7)(S|_{V_1}),
\end{equation}
where $\theta_3$ and $\theta_7$ are listed on pp. 523--524 of \cite{Sr}: $\theta_3$ is irreducible of degree $(q^2+1)/2$, and $-\theta_7$ is irreducible of 
degree $(q^2-1)/2$.  

First suppose that $(u|v)=0$ (but $v \notin \langle u \rangle_{\F_q}$). The number of possibilities for such $(\alpha,u,\beta,v)$ is 
$$(q^{2n}-1)(q^{2n-1}-q)(q-1)^2/4.$$
We may assume that $V_1$ has a basis $(u_1=u,v_1=v,u_2,v_2)$ in which $(\cdot|\cdot)$ has Gram matrix $\begin{pmatrix}0 & I_2\\-I_2 & 0 \end{pmatrix}$. The matrix of
$S|_{V_1}$ in this basis is $\begin{pmatrix}I_2 & X\\0 & I_2\end{pmatrix}$, where $X := \mathrm{diag}(\alpha,\beta)$. If $q \equiv 1 \pmod 4$, then the 
quadratic form $(x,y) \mapsto \alpha x^2 +\beta y^2$ is isotropic, and hence we see (using the conjugation action of the Levi subgroup 
$$\mathrm{Stab}_{\Sp(V_1)}\bigl(\langle u_1,v_1\rangle_{\F_q}, \langle u_2,v_2\rangle_{\F_q}\bigr) \cong \GLtwoq$$
of $\Sp(V_1)$) that $S$ belongs to class $A_{31}$, one of the two classes of double transvections in \cite{Sr}. If $q \equiv 3 \pmod 4$, then the 
quadratic form $(x,y) \mapsto \alpha x^2 +\beta y^2$ is anisotropic, and hence $S$ belongs to class $A_{32}$, the other class of double transvections in \cite{Sr}. 
Thus the probability that $S$ belongs to class $A_{31}$, respectively $A_{32}$, is 
$$\frac{q^{2n-1}-q}{q^{2n}-1}, \mbox{ resp. }0,$$
when $q \equiv 1 \pmod4$, and  
$$0, \mbox{ resp. }\frac{q^{2n-1}-q}{q^{2n}-1},$$
when $q \equiv 3 \pmod4$.

Now assume that $(u|v) \neq 0$. Rescaling $u$, which leads to a factor $q-1$ in counting the number of possibilities, we may assume $(u|v)=1$. 
For a fixed $0 \neq u \in V$, the number of such $v$ is $(q^{2n}-q^{2n-1})/(q-1) = q^{2n-1}$.
We may assume that $V_1$ has a basis $(u,v,u',v')$ in which $(\cdot|\cdot)$ has Gram matrix $\mathrm{diag}(\Gamma,\Gamma)$, 
where $\Gamma=\begin{pmatrix}0 & 1\\-1 & 0 \end{pmatrix}$. The matrix of
$S|_{V_1}$ in this basis is $\mathrm{diag}(Y,I_2)$, where $Y:=\begin{pmatrix}1-\alpha\beta & \alpha \\ -\beta & 1\end{pmatrix}$.

Consider the case $\alpha\beta=4$. Then $-Y =  \begin{pmatrix}3 & -\alpha \\ 4/\alpha & -1\end{pmatrix}$ is conjugated by the element
$\begin{pmatrix}\sqrt{\alpha} & 0\\ 2/\sqrt{\alpha} & 1/\sqrt{\alpha}\end{pmatrix} \in \Sp_2(q)$ to $\begin{pmatrix}1 & -1\\0 & 1\end{pmatrix}$. Hence
$S$ belongs to class $D_{21}$ in \cite{Sr} if $q \equiv 1 \pmod4$, and class $D_{22}$ if $q \equiv 1 \pmod4$, and the probability of this event is
$$\frac{2q^{2n-1}}{q^{2n}-1}.$$

Now we consider the case $\alpha\beta \neq 4$. Then $Y \in \Sp_2(q)$ has trace $2-\alpha\beta \neq -2$. By Lemma \ref{lem:sq2}, $Y$ is conjugate to
$\mathrm{diag}(\gamma^i,\gamma^{-i})$ and then $S$ belongs to class $C_3(i)$ in \cite{Sr}, or to 
$\mathrm{diag}(\eta^j,\eta^{-j})$ and then $S$ belongs to class $C_1(j)$ in \cite{Sr}, each possibility completely determined by $\alpha\beta$.
More precisely, when $q \equiv 1 \pmod 4$, $S$ belongs to 
$(q-5)/4$ classes $C_3(i)$, all with $2|i$, and $(q-1)/4$ classes $C_1(j)$, all with $2 \nmid j$, and the probability that $S$ belongs to each of these classes is 
$$\frac{2q^{2n-1}}{q^{2n}-1}.$$
When $q \equiv e \pmod 4$, $S$ belongs to 
$(q-3)/4$ classes $C_3(i)$, all with $2 \nmid i$, and $(q-3)/4$ classes $C_1(j)$, all with $2 \mid j$, and the probability that $S$ belongs to each of these classes is 
$$\frac{2q^{2n-1}}{q^{2n}-1}.$$

\smallskip
Using \eqref{eq:weil1} and \cite{Sr}, we can list the values of $\omega(S)$ as
\begin{equation}\label{eq:weil2}
  -q^{n-1}\kappa\sqrt{\kappa q},~q^{n-1}\kappa\sqrt{\kappa q},~q^{n-1},~-q^{n-1},~-(-1)^jq^{n-1},~(-1)^iq^{n-1},~\kappa q^{n-1}, \mbox{ resp. } \kappa q^{n-1},
\end{equation}  
when $S$ belongs to class $C=A_{21}$, $C^*=A_{22}$, $A_{31}$, $A_{32}$, $C_1(j)$, $C_3(i)$, $D_{21}$, or $D_{22}$, respectively. Pairing these values with
the above determined probabilities, we obtain the following main result of this section:

\begin{theorem}\label{sp-odd-prob1}
Let $q$ be any odd prime power, $n \geq 2$, and let $\omega$ denote a reducible Weil character of degree $q^n$ of $\Spnq$.
Fix a conjugacy class $C$ of transvections in $\Spnq$, and 
for any conjugacy class $T$ of $\Spnq$, let $p_2(T)$ be the probability that the product of two random elements from $C$ belongs to $T$. Then 
for any integer $r \geq 0$, 
$$\sum_{T}\omega(T)^r p_2(T)$$
is equal to
$$\frac{q^{(n-2)r}}{q^{2n}-1} \cdot\Bigl( 2q^{2r}+\frac{q-1}{2}q^{3r/2}+(-1)^r\frac{q-5}{2}q^{3r/2}+q^r(q^{2n}-q) \Bigr)$$
if $q \equiv 1 \pmod4$, and 
$$\frac{q^{(n-2)r}}{q^{2n}-1} \cdot\Bigl( (-1)^r\frac{q-3}{2}(-q)^{3r/2}+\frac{q+1}{2}(-q)^{3r/2}+(-q)^r(q^{2n}-q) \Bigr)$$
if $q \equiv 3 \pmod4$ (and with a suitable choice of $\sqrt{-q}$).
\end{theorem}

We will also need the following variant of Theorem \ref{sp-odd-prob1}:

\begin{theorem}\label{sp-odd-prob2}
Let $q$ be any odd prime power, $n \geq 2$, and let $\omega$ denote a reducible Weil character of degree $q^n$ of $\Spnq$.
For any conjugacy class $T$ of $\Spnq$, let $\tilde p_2(T)$ be the probability that the product of two random transvections belongs to $T$. Then 
for any even integer $r \geq 0$, 
$$\sum_{T}\omega(T)^r \tilde p_2(T) = \frac{q^{(n-2)r}}{q^{2n}-1} \cdot\Bigl( q^{2r}+(q-2)\kappa^{r/2}q^{3r/2}+q^r(q^{2n}-q) \Bigr),$$
where $\kappa:=(-1)^{(q-1)/2}$.
\end{theorem}

\begin{proof}
Consider the product $S=T(\alpha,u)T(\beta,v)$ of two random transvections, i.e. we now assume only that 
$\alpha,\beta \in \F_q^\times$. If $S=\mathrm{Id}$, then $\omega(S)=q^n$. 
If $S$ has rank $1$, then 
$S$ is a transvection, and using \eqref{eq:weil1}, \eqref{eq:weil2} we have 
$$\omega(S)^2 = \kappa q^{2n-1}.$$ 
If $S$ has rank $2$, then the same analysis as in 
the proof of Theorem \ref{sp-odd-prob1} shows that $S$ belongs to one of the classes $A_{31}$, $A_{32}$, $C_1(j)$, $C_3(i)$, $D_{21}$, or $D_{22}$ (but now without
parity restrictions on $i$ and on $j$).
Again using \eqref{eq:weil1}, \eqref{eq:weil2}, we have 
$$\omega(S)^2 = q^{2n-2}.$$ 
Since $2|r$, the statement follows from Theorem \ref{sp-transv}.  
\end{proof}

\end{document}